\documentclass[11pt,a4paper,reqno]{amsart}

\usepackage[T1]{fontenc}
\usepackage[utf8]{inputenc}
\usepackage{newtxtext,newtxmath}
\usepackage{microtype}
\usepackage{geometry}
\usepackage{mathtools}
\usepackage{enumitem}
\usepackage{xcolor}
\usepackage{hyperref}
\usepackage[nameinlink,capitalize,noabbrev]{cleveref}

\hypersetup{
  colorlinks=true,
  linkcolor=blue!45!black,
  citecolor=blue!45!black,
  urlcolor=blue!45!black,
  pdftitle={A Real-Analytic Counterexample to Quasiconcavity for an Ornstein--Uhlenbeck Capacitary Potential},
  pdfauthor={Lei Qin}
}
\allowdisplaybreaks
\numberwithin{equation}{section}

\newtheorem{theorem}{Theorem}[section]
\newtheorem{proposition}[theorem]{Proposition}
\newtheorem{lemma}[theorem]{Lemma}

\theoremstyle{remark}
\newtheorem{remark}[theorem]{Remark}

\newcommand{\R}{\mathbb{R}}

\newcommand{\Om}{\Omega}

\newcommand{\Ccurv}{\mathcal{C}}

\newcommand{\ep}{e_{\perp}}

\newcommand{\capa}{\mbox{\rm cap}}

\title[A counterexample to quasiconcavity for the Gaussian $p$-Capacity potentials in convex rings]
{A counterexample to quasiconcavity for the\\ Gaussian $p$-Capacity potentials in convex rings }

\author{Lei Qin}
\address{Department of Mathematics and Statistics, University of Montreal, Montreal, Quebec, Canada}
\email{lei.qin@umontreal.ca}


\begin{document}

\begin{abstract}
We construct a smooth, strictly convex ring that admit a unique solution of Gaussian $p$-Capacity problem whose super-level sets are nonconvex. This result provide a negative answer.
\end{abstract}

\maketitle

\section{Introduction}
A class of profoundly significant problems involves the study of the convex properties of level sets for solutions to partial differential equations in the class of convex rings. There are many new techniques that have appeared and been developed over these years (microscopic and macroscopic techniqque). The microscopic approach studies the second fundamental form of level hypersurfaces through differential inequalities for suitable curvature-based test functions. The central tool is the constant rank
theorem, which controls the degeneracy of this second fundamental
form by means of the strong maximum principle. See, for instance, \cite{BGMX2011}. The macroscopic approach relies on geometric envelopes and comparison
principles. Showing that the envelope is a viscosity sub(sup)-solution with
appropriate boundary behavior allows one to identify it with the
original solution, thereby establishing quasi-concavity. See for instance, \cite{BLS2009}. A classical reference text is the book by Kawohl \cite{Kawohl1985}. For further historical background, we refer the reader to these works and the references therein.

For the classical $p$ ($1<p<+\infty$) harmonic capacitary potential in a convex ring, the
convexity of level sets is a fundamental result of Lewis
\cite{Lewis1977}. In the planar case, Longinetti \cite{Longinetti1983} derived a linear elliptic
equation for the radius of curvature of the level curves and obtained a
boundary curvature estimate from the maximum principle. These results illustrate how the convexity of the
boundary components is inherited by the intermediate level
hypersurfaces. It is therefore natural to ask whether this convexity property
persists when the Lebesgue measure in the capacitary energy is
replaced by the Gaussian measure. The purpose of this paper is to show that the analogous quasiconcavity of Gaussian $p$-Capacity potentials in convex rings fails. 

Let $1<p<\infty$, define
$$
-L_{p} (u):=-\text{div}(|\nabla u|^{p-2} \nabla u)+(x,\nabla u)|\nabla u|^{p-2},\quad x\in \R^n.
$$
Here $\nabla$ denotes the Euclidean gradient. In particular, when $p=2$, $L_2$ is the Ornstein--Uhlenbeck operator. Very recently, the operator $L_p$ have attracted increasing attention. Particularly in the geometric properties of solutions and Brunn–Minkowski-type inequalities for the associated variational functionals. The first Dirichlet eigenvalue problem of the operator $L_p$ has been studied from both elliptic and parabolic perspectives, see, for instance, \cite{Carbotti2026,Andrea-Paolo2024,CQS2026,
CQS2026-2,Qin2026}. These reference give the positive answer. By contrast, recent work on the torsion problem has revealed obstructions to several natural Brunn–Minkowski-type inequalities in Gaussian space. See for instance \cite{Alina2026,SS2026}. In this paper, we concern the capacity problem for the Ornstein–Uhlenbeck operator and its nonlinear extension to Gaussian $p$-capacity. Our focus is on the geometry of capacitary potentials in convex rings. 

Let $\Om_1\Subset\Om_0$ be bounded Lipschitz domains in $\R^n$. We consider the solution of the Dirichlet problem
\begin{equation}\label{eq:OU-problem-intro}
\begin{cases}
 L_{p}u=0 & \text{in }\ \Omega:=\Om_0\setminus\overline{\Omega}_1,\\
 u=0 & \text{on }\partial\Om_0,\\
 u=1 & \text{on }\partial\Om_1.
\end{cases}
\end{equation}
The Gaussian capacity of $\Omega_1$ relative to $\Omega_0$, $\capa_\gamma(\Omega_1,\Omega_1)$ is defined by
$$
\capa_\gamma(\Omega_1,\Omega_0)=\inf\left\{ \int_\Omega|\nabla u|^2d\gamma_n:\ u\in W_0^{1,p}(\Omega_0,\gamma_n),\ u\geq 1\ \mathrm{in }\  \Omega_1\right\}.
$$
Here $\gamma_n$ denotes the Gaussian probability measure in $\R^n$, given by
$$
\gamma_n(A)= \frac{1}{(2\pi)^{n/2}} \int_{A} e^{-|x|^2/2} dx\,,\quad\text{for any measurable set }A\,.
$$
A standard variational method shows that there exists a (nontrivial) positive solution $u$ such that
$$
\capa_\gamma(\Omega_1,\Omega_0)=\int_\Omega|\nabla u|^2d\gamma_n.
$$
And the solution is unique, up to multiplication by a positive constant. The standard regularity theory in \cite{Evans,GT} shows that $u\in C^{\infty}(\Omega)\cap C^0(\overline{\Omega})$ and $u$ satisfies the boundary value problem \eqref{eq:OU-problem-intro}. Moreover, if the boundary of $\Omega$ is smooth, the solution $u$ is smooth up to the boundary.

We call the minimizers $u$ the Gaussian $p$-capacitary potentials. we extend $u$ to $1$ in $\overline{\Omega}_1$, the extended function is called quasiconcave if every superlevel set is convex. Our main results is as follows.

\begin{theorem}\label{T1}
Let $p\in(1,+\infty)$ and $n\geq 2$. Then there exist bounded domains $\Om_1\Subset\Om_0 \subset \R^n$  with smooth strictly convex doundaries such that the unique solution $u$ of \eqref{eq:OU-problem-intro} has a nonconvex superlevel set. More precisely, after extending $u$ by $1$ in $\Omega_1$, the set 
\[
\left\{ x\in \Omega_0:\ u(x)\geq \frac{1}{2}\right\}
\]
is not convex.
\end{theorem}

One of the crucial steps in the proofs of Theorem \ref{T1} is an equivalent radial-graph formulation of equation \eqref{eq:OU-problem-intro}. More precisely, by parametrizing the level hypersurfaces as radial graphs, the equation \eqref{eq:OU-problem-intro} is transformed into a analytic partial differential equation to which the Cauchy–Kowalevski theorem applies. A suitable choice of real-analytic Cauchy data then yields two strictly convex endpoint hypersurfaces and a nonconvex intermediate level hypersurface. The argument is direct and intuitive, and may be of independent interest. 

The proof is presented in three stages. We begin with the two-dimensional Gaussian capacity problem corresponding to $p=2$. In this case the mechanism responsible for the counterexample is particularly transparent. This planar argument provides the geometric intuition. Then, we establish the higher-dimensional counterexample. Here, the scalar curvature used in the planar case must be replaced by the second fundamental form, whose eigenvalues determine the convexity of the corresponding level hypersurface. Finally, the result is generalized to the weighted $p$-Laplacian case.

The organization of the paper is as follows. In Section 2, we list some basic facts and notations. In Section 3, we construct the counterexample for Gaussian Capacity potential in dimension 2. In Section 4, we construct the counterexample in $\R^n$ and generalize the result of Gaussian Capacity to the case of $p$-Laplacian.

\section{Preliminaries}

\subsection{Basic notation}
Let $\mathbb{R}^n$ be the $n$-dimensional Euclidean space, $x$ is a point in $\R^n$, $|x|$ is the Euclidean norm of $x$ and $dx$ denotes integration with respect to Lebesgue measure in $\R^n$.  The unit sphere in $\mathbb{R}^n$ is denoted by $\mathbb{S}^{-1}$. Let $x_0\in \R^n$ and $r>0$, we define the distance function $d(x):=r-|x-x_0|$, for each $x\in B_r(x_0)$. 

If $K$ is a convex body in $\mathbb{R}^n$ with the origin in its exterior, the radial function $\rho_{K}$ is defined by
\begin{equation}\nonumber
\rho_{K}(x)=\max \{\lambda: \lambda x \in K\}, \ \ x\in \mathbb{R}^n\setminus \{0\}.
\end{equation}

\subsection{Radial coordinates in Dimension 2}

Let $r\in C^2(\mathbb{S}^{1}\times (-\delta,\delta))$ satisfy 
\begin{equation}\label{eq:r-signs-preliminary}
 r(\theta,s)>0,
 \qquad
 r_s(\theta,s) <0,\quad \theta\in[0,2\pi)\times (-\delta,\delta)
\end{equation}
where $\delta>0$ is a real number. Define
\[
 e(\theta)=(\cos\theta,\sin\theta),
 \qquad
 \ep(\theta)=(-\sin\theta,\cos\theta),
 \qquad \theta\in [0,2\pi).
\]
In Dimension 2, the position vector can be parameterized in term of the radial function as follows
\begin{equation}\label{eq:X-def}
X(\theta,s)=r(\theta,s)e(\theta).
\end{equation}
Then, \eqref{eq:r-signs-preliminary} allows us to define a function $u$ locally by
\begin{equation}\label{eq:v-def}
u \bigl(r(\theta,s)e(\theta)\bigr)=s.
\end{equation}
Writing a point as $x=\rho e(\theta)$, the defining relation is
\begin{equation}\label{eq:implicit-relation}
\rho=r\bigl(\theta,u(\rho,\theta)\bigr).
\end{equation}
For fixed $s$, the curve $\theta\mapsto X(\theta,s)$ is oriented counterclockwise.  Since
\[
 X_\theta=r_\theta e+r\ep,
 \qquad
 X_{\theta\theta}=(r_{\theta\theta}-r)e+2r_\theta\ep,
\]
its curvature is given by
\begin{equation}\label{eq:curvature}
 \kappa(\theta,s)
 =\frac{\Ccurv(\theta,s)}{(r^2+r_\theta^2)^{3/2}},
 \qquad
 \Ccurv=r^2+2r_\theta^2-r r_{\theta\theta}.
\end{equation}
In particular, $\Ccurv>0$ at every point implies that the curve is  strictly convex. While
$\Ccurv<0$ at one point excludes convexity of curve.

\subsection{Cauchy-Kowalevski Theorem}

We now recall the analytic existence and uniqueness theorem which is used to construct counterexample. Please refer to \cite{CK-perturbed,Nirenberg1972,Walter1985} and the references cited therein.

\begin{proposition}\label{0911}
Let $x\in \R^n$, let $m\geq 1$, and suppose that $F$ is real analytic in all variables near the initial date. Consider an equation 
\[
\partial^{m}_t u=F(x,t,\{\partial_x^{\alpha } \partial^{j}_t u:\ 0\leq   j<m,|\alpha|+j\leq m\}) 
\]
with real-analytic Cauchy data
\[
\partial^{j}_t u(x,0)=\phi_j(x),\quad 0\leq j <m.
\]
Then, near every point of the initial hypersurface $t=0$ and $x$, there exists a unique real-analytic solution $u$. 
\end{proposition}

The continuous dependence of the solution $u$ with respect the perturbed Cauchy data is established in \cite[Theorem~3 and Remark~3]{Walter1985}. A more general abstract version of this result is available in \cite[Theorem 3.1]{CK-perturbed}. We shall use the following real-valued version of
Walter's theorem. It follows by extending the real-analytic
coefficients and Cauchy data to a fixed complex neighborhood,
applying Walter~\cite[Theorem~3 and Remark~3]{Walter1985},
and then restricting the resulting solution to the real
variables. 

We state below the version needed for our purposes.

\begin{proposition}\label{0909}
Let $x_0\in \R^n$ and $r>0$. let $\eta>0$ define 
\[
G_\eta:=\left\{ (t,x):|t|\leq \eta d(x),\quad x\in B_{r}(x_0) \right\},\quad d(t,x):=d(x)-\frac{|t|}{\eta}.
\]
Consider the quasi-linear system of the first order as follows: 
\[
u_t=\sum_{j=1}^n B_j(u)u_{x_j}+c(u),\quad (t,x,u)\in\ G_\eta \times (0,C_0)
\]
with the Cauchy data
\[
u(x,0)=\phi(x)\quad \mathrm{in}\ \R^n.
\]
If $B_j(u)$ and $c(u)$ are analytic with respect to $u$, and
\[
\sqrt{d(t,x)}|c|\leq \gamma,\quad d(t,x)|c(u)-c(\upsilon)|\leq \gamma'|u-\upsilon|,
\]
and
\[
|B_j|\leq \beta_j,\quad  \sqrt{d(t,x)}|B_j(u)-B_j(\upsilon)|\leq \beta'_j |u-\upsilon|,\quad j=1,\ldots, n.
\]
If there exists $\eta>0$ satisfies
\[
2\eta \sqrt{r}\left(\sum_{j}\beta_j+\gamma\right)<C_0,\quad \eta \left((3\sqrt{3}+2)\sum_{j}\beta_j\gamma+2\sum_{j}\beta_j+3\sqrt{3} \gamma\right)\leq 1,\quad 3\eta  \sum_{j}\beta_j'\leq 1.
\]
Then the quasilinear system has a unique solution $u$ in $G_\eta:=\left\{(t,x):\ |t|\leq \eta d(x),\ \forall x\in B_r(x_0)\right\}$. Moreover, if we perturb Cauchy data $\phi$ as $\phi^{\varepsilon}$, and all bounds are uniform with respect to sufficiently small $\varepsilon>0$. In addition, we assume that there exists a constant $c_0>0$ such that
\[
|\phi^{\varepsilon}-\phi| \leq c_0\varepsilon,\quad \forall\ x\in \R^n.
\]
Then $u^\varepsilon \rightarrow u$ in $C^0 (G_\eta)$, as $\varepsilon \rightarrow 0$.
\end{proposition}

\section{The planar counterexample}

In this section, we prove Theorem \ref{T1} in the case $n=2$ and $p=2$. We begin by deriving an equivalent radial-graph formulation of equation \eqref{eq:OU-problem-intro}. 

\subsection{The radial-graph equation}

Throughout this section, all angular derivatives of $u$ are taken with $\rho$  fixed. 

\begin{proposition}\label{lem:inverse-derivatives}
At a point satisfying $\rho=r(\theta,s)$, the function defined by \eqref{eq:v-def} satisfies
\begin{align*}
 u_\rho&=\frac{1}{r_s},
 &
 u_{\rho\rho}&=-\frac{r_{ss}}{r_s^3},
 \label{eq:rho-derivatives}\\[1mm]
 u_\theta&=-\frac{r_\theta}{r_s},
 &
 u_{\theta\theta}
 &=-\frac{r_{\theta\theta}}{r_s}
   +\frac{2r_\theta r_{s\theta}}{r_s^2}
   -\frac{r_\theta^2r_{ss}}{r_s^3}.
\end{align*}
\end{proposition}

\begin{proof}[\bf Proof.]
Fixed $\theta\in [0,2\pi)$, we differentiate \eqref{eq:implicit-relation} with respect to $\rho$, then
\[
 1=r_s u_\rho,\quad 0=r_{ss}u_\rho^2+r_su_{\rho\rho}.
\]
This proves the first equality. Differentiating \eqref{eq:implicit-relation} with respect to $\theta$ with fixed $\rho$, we have
\begin{equation}\label{eq:first-angular-implicit}
 0=r_\theta+r_su_\theta.
\end{equation}
Thus $u_\theta=-r_\theta/r_s$.  Differentiating
\eqref{eq:first-angular-implicit} once more with the fixed $\rho$,
due to the dependence $s=u(\rho,\theta)$, we deduce that
\[
 \left.\frac{\partial}{\partial\theta}r_\theta
 \right|_{\rho}
 =r_{\theta\theta}+r_{\theta s}u_\theta,
 \qquad
 \left.\frac{\partial}{\partial\theta}r_s
 \right|_{\rho}
 =r_{s\theta}+r_{ss}u_\theta.
\]
Consequently,
\begin{align*}
 0
 &=r_{\theta\theta}+r_{\theta s}u_\theta
   +(r_{s\theta}+r_{ss}u_\theta)u_\theta
   +r_su_{\theta\theta}\\
 &=r_{\theta\theta}+2r_{s\theta}u_\theta
   +r_{ss}u_\theta^2+r_su_{\theta\theta}.
\end{align*}
Since $u_\theta=-r_\theta/r_s$, the proof is completed.
\end{proof}

Using polar coordinates, we have
\[
 \Delta u=u_{\rho\rho}+\frac1\rho u_\rho
 +\frac1{\rho^2} u_{\theta\theta},
 \qquad
 x\cdot\nabla u=\rho u_\rho.
\]

\begin{proposition}\label{prop:evolution}
Assume that \eqref{eq:r-signs-preliminary} holds. The function $u$ defined by \eqref{eq:v-def} satisfies
\begin{equation}\label{eq:OU-v}
 -\Delta u+x\cdot\nabla u=0,\quad in\quad  \Omega
\end{equation}
if and only if $r$ satisfies
\begin{equation}\label{eq:evolution}
 (r^2+r_\theta^2)r_{ss}
 -2r_\theta r_s r_{s\theta}
 +r_s^2r_{\theta\theta}
 +(r^3-r)r_s^2=0.
\end{equation}
Equivalently, 
\begin{equation}\label{eq:evolution-normal}
 r_{ss}
 =\frac{2r_\theta r_sr_{s\theta}
       -r_s^2r_{\theta\theta}
       -(r^3-r)r_s^2}
      {r^2+r_\theta^2}.
\end{equation}
\end{proposition}

\begin{proof}[\bf Proof.]
Since $s\mapsto r(\theta,s)$ is strictly decreasing for every \(\theta\), and \(r\) is \(2\pi\)-periodic in \(\theta\), each point of the ring domain admits a unique representation \(x=r(\theta,s)e(\theta)\). Hence \(u(x):=s\) is well defined.

Since $\rho=r(\theta,s)$, Proposition \ref{lem:inverse-derivatives} and \eqref{eq:OU-v} give
\begin{align*}
 0={}&-\frac{r_{ss}}{r_s^3}
 +\frac{1}{rr_s}
 +\frac{1}{r^2}
 \left(
 -\frac{r_{\theta\theta}}{r_s}
 +\frac{2r_\theta r_{s\theta}}{r_s^2}
 -\frac{r_\theta^2r_{ss}}{r_s^3}
 \right)
 -\frac{r}{r_s}.
\end{align*}
Multiplying by $-r^2r_s^3$ gives \eqref{eq:evolution}.
Since $r^2+r_\theta^2>0$, this is equivalent to \eqref{eq:evolution-normal}.
\end{proof}

\subsection{Construction of a curve with zero-curvature point}

Here, we construct a curve that both has zero-curvature points and satisfies the equation \eqref{eq:evolution}.

Let
\[
H(\theta)=10-2\cos 2\theta,\quad Q(\theta)=-H(\theta)+7\sin \theta -\sin 3\theta,\quad \forall\ \theta\in [0,2\pi).
\]
Obviously, 
\begin{equation}\label{0902-1}
H(\theta)\geq8,
 \qquad
 Q(\theta)\leq-4
 \quad\text{for every }\theta\in\mathbb [0,2\pi).
 \end{equation}
The Cauchy-Kowalevski theorem yields the following Lemma.

\begin{lemma}\label{Application-CK}
There exists a unique solution $r(\theta,s)$ which satisfies equation \eqref{eq:evolution} 
with Cauchy data
\[
 r(\theta,0)=H(\theta),
 \qquad
 r_s(\theta,0)=Q(\theta).
\]
\begin{enumerate}
\item After decreasing $\delta>0$ if necessary,
\[
 r>0,
 \qquad
 r_s<0
 \quad\text{in }\mathbb [0,2\pi)\times [-\delta,\delta].
\]
Moreover, $r$ is a real-analytic function jointly in $\R \times [-\delta,\delta]$ with period $2\pi$ .

\item If $s=0$, we have
\begin{equation}\nonumber
\Ccurv(\theta,0)
 =
 48\sin^2\theta\,(5+\cos^2\theta)\geq 0,\quad \forall\ \theta\in [0,2\pi),
\end{equation}
where $\Ccurv$ is defined by \eqref{eq:curvature}. 

\item For each $\theta_*\in\{0,\pi\}$, we have
\[
\Ccurv (\theta_*,0)
 =\Ccurv_\theta(\theta_*,0)
 =\Ccurv_s(\theta_*,0)
 =\Ccurv_{\theta s}(\theta_*,0)=0,
\]
and
\[ \Ccurv_{\theta\theta}(\theta_*,0)=576,
 \qquad
 \Ccurv_{ss}(\theta_*,0)=1472.
\]
This implies that $D_{(\theta,s)}^2 \Ccurv(\theta_*,0)>0$.
\end{enumerate}
\end{lemma}

\begin{proof} [\bf Proof. ]
By \eqref{eq:r-signs-preliminary} and \eqref{eq:X-def}, we know that if $\theta\in [0,2\pi)$, we have a closed curve in dimension two. Now, we apply Cauchy-Kowalevski theorem to equation \eqref{eq:evolution} with the above Cauchy data, the existence and uniqueness of solution is proved. More precisely, for any $\theta_0\in \R$, there exist an open neighborhood $U_{\theta_0}$, a number $\delta_{\theta_0}>0$ and an unique real-analytic solution $u$ in $U_{\theta_0}\times (-\delta_{\theta_0},\delta_{\theta_0})$.  Since $[0,2\pi]$ is compact, finitely many of the sets
\(U_{\theta_0}\) cover $[0,2\pi]$.

The uniqueness of the Cauchy--Kowalevski theorem implies that
the local solutions agree on overlaps, and hence they glue to a
real-analytic solution on
\[
[0,2\pi]\times(-\delta,\delta).
\]
By \eqref{0902-1} and smoothness of function, after decreasing $\delta>0$, we have
\[
r>0,\quad r_{s}<0, \quad {in}\ \mathbb [0,2\pi]\times [-\delta,\delta].
\]
Since the initial data has period $2\pi$ and the equation \eqref{eq:evolution} is invariant under the translation $\theta\rightarrow \theta+2\pi$, we extend the solution $2\pi$-periodically to $\R$. Thus, $(1)$ holds.

In the following, we only consider the case $s=0$. By \eqref{eq:curvature} and \eqref{0902-1}, we have
\begin{equation}\nonumber
\Ccurv(\theta,0)=r^2+2r_\theta^2 -r r_{\theta \theta}=H^2+2H_\theta^2 -H H_{\theta \theta}=48 \sin^2 \theta (5+\cos^2\theta).
\end{equation}
Therefore, the curvature of the initial curve is convex and its curvature vanishes only at $\theta_{*}=0,\pi$.

By \eqref{eq:curvature}, we have
\begin{equation}\label{E:00088}
\left\{
\begin{aligned}
& \Ccurv_{\theta}=2rr_\theta +3r_\theta r_{\theta\theta}-rr_{\theta \theta \theta},\\
&\Ccurv_{s}=2rr_s+4r_\theta r_{\theta s}-r_s r_{\theta \theta}-rr_{\theta \theta s},
\end{aligned}
\right.
\end{equation}
and
\begin{equation}\label{2-Curvature}
\left\{
\begin{aligned}
& \Ccurv_{\theta\theta}=2r_\theta^2+2rr_{\theta \theta}+3r_{\theta \theta}^2+2r_\theta r_{\theta \theta \theta }-rr_{\theta \theta \theta \theta},\\
&\Ccurv_{\theta s}=2r_{s}r_\theta +2rr_{\theta s}+3r_{\theta s} r_{\theta \theta}+3r_\theta r_{\theta \theta s}-r_s r_{\theta \theta \theta}-rr_{\theta \theta \theta s},\\
&\Ccurv_{ss}=2 r_s^2+(2r-r_{\theta \theta})r_{ss}+4r_{\theta s}^2+4r_\theta r_{\theta ss}-2r_s r_{\theta \theta s}-rr_{\theta \theta ss}.
\end{aligned}
\right.
\end{equation}
By a direct calculation, at $\theta_{*}$ we have
\begin{equation}\label{0902-2}
\left\{
\begin{aligned}
& 
H=8,\quad H_\theta=0,\quad H_{\theta \theta}=H=8,\quad H_{\theta \theta \theta}=0,\quad H_{\theta \theta \theta \theta}=-32;\\
& Q=-8,\quad Q_{\theta }=\pm 4,\quad Q_{\theta \theta}=Q=-8, \quad Q_{\theta \theta\theta}=5Q_\theta=\pm 20.
\end{aligned}
\right.
\end{equation}
Therefore, at the point $(\theta_{*},0)$, we have
\[
\Ccurv_\theta=0=\Ccurv_{s}=\Ccurv_{\theta s},\quad \Ccurv_{\theta \theta}=576.
\]
In the following, we prove the following equality
\begin{equation}\label{3-Curvature}
H^2 \Ccurv_{ss}+ Q^2 \Ccurv_{\theta \theta}=2H^4 Q^2_\theta,
\end{equation}
at point $(\theta_{*},0)$, which implies $\Ccurv_{ss}(\theta_{*},0)=1472$.

By Proposition \ref{prop:evolution}, we have
\begin{equation}\nonumber
 r_{ss}
 =\frac{2r_\theta r_sr_{s\theta}
       -r_s^2r_{\theta\theta}
       -(r^3-r)r_s^2}
      {r^2+r_\theta^2}.
\end{equation}
When $s=0$, we have
\begin{equation}\label{0903-1}
r_{ss}=\frac{2H_\theta Q Q_{\theta}-Q^2 H_{\theta \theta}-(H^3-H)Q^2}{H^2+H_\theta^2}.
\end{equation}
At point $(\theta_{*},0)$, by \eqref{0902-2} and \eqref{0903-1}, we have
\begin{equation}\nonumber
\left\{
\begin{aligned}
& r_{ss\theta}=-\frac{(H^2-1)2QQ_{\theta}}{H}=\pm 504,\\
&r_{ss\theta \theta}=-HQ^2-2HQ_\theta^2+\frac{5Q^2}{H}+\frac{4Q^2_\theta}{H}-\frac{H_{\theta\theta\theta\theta}Q^2}{H^2}.
\end{aligned}.
\right.
\end{equation}
By \eqref{2-Curvature}, \eqref{0902-2} and the above equalities, we prove \eqref{3-Curvature}. The proof is completed.
\end{proof}

\begin{remark}
Proposition \ref{prop:evolution} and Lemma \ref{Application-CK} yield a real-analytic solution $r$ of equation~\eqref{eq:evolution}, whose level curves form a nested family and whose central member has zero curvature at certain points. This construction therefore produces a convex ring with strictly convex boundary whose Gaussian Capacitary potential possesses an interior level curve with isolated points of zero curvature. This indicates that the constant-rank theorem may fails in the Gaussian setting.
\end{remark}

\subsection{Construction of a curve with negative curvature point}

First, we notice that there exist many solutions to the equation
\[
-\Delta w +(x,\nabla w)=0,\quad \mathrm{in}\ \R^n.
\]
For example, we have
\[
w(x)=c_1 +c_2 \int_{0}^{x_2} e^{t^2/2}dt,\quad x=(x_1,x_2,\ldots,x_n)\in \R^{n}.
\]
In general, we know
\[
w(x)=c_1+c_2 \int_{0}^{(x,\xi)} e^{t^2/2} dt,\quad \forall\ \xi \in \mathbb{S}^{n-1}.
\]
In the following, we choose
\begin{equation}\label{Auxiliary-func}
w(x)=-\int_{0}^{x_2} e^{t^2/2}dt,\quad x=(x_1,x_2)\in \R^2.
\end{equation}

\begin{lemma}\label{3.2-1}
There exist $\delta_{*}\in (0,\delta)$, $\varepsilon_{*}>0$, and a jointly real-analytic function
\[
r^{\varepsilon}(\theta,s)=r(\theta,s,\varepsilon),\quad (\theta,s,\varepsilon)\in [0,2\pi)\times [-\delta_{*},\delta_{*}]\times (-\varepsilon_{*},\varepsilon_{*}).
\]
satisfying equation \eqref{eq:evolution}. Moreover, $r^{\varepsilon}$ can be extended to $\R$ with $2\pi$ period and
\[
r^{\varepsilon}>0,\quad r_s^{\varepsilon}<0,
\]
and $r^{\varepsilon}\rightarrow r^{0}$ in $C^{k}$ uniformly in $[0,2\pi)\times [-\delta_{*},\delta_{*}]$ for every fixed $k$. 
\end{lemma}

\begin{proof}[\bf Proof.]
By Lemma \ref{Application-CK}, there exist $\delta>0$, $r^0$ and $u^{0}$ such that
\begin{equation}\label{0829-1}
u^{0}(r^0(\theta,s)e(\theta)):=s,\quad \forall\ (\theta,s)\in \R\times [-\delta,+\delta],
\end{equation}
where $r^0$ satisfies \eqref{eq:evolution}. By Proposition \ref{prop:evolution}, we know that $u^0$ satisfies \eqref{eq:OU-v}.

Consider function
\[
u^{\varepsilon}=u^0+\varepsilon w.
\]
Obviously,
\[
L_\gamma (u+\varepsilon w)=0.
\]
Define
\[
G(\theta,\rho,s,\varepsilon)=u^{\varepsilon}(\rho e(\theta))-s,\quad (\theta,\rho,s,\varepsilon)\in \R\times \mathbb{R}\times \mathbb{R}\times \mathbb{R}.
\]
Therefore, when $\varepsilon=0$, we have
\begin{equation}\label{0829-2}
G(\theta,r^{0}(\theta,s),s,0)=0.
\end{equation}
Differentiating \eqref{0829-1} with respect to $s$, we have
\[
u_{\rho}^0(r^{0}(\theta,s)e(\theta)) r^{0}_s(\theta,s)=1.
\]
By Lemma \ref{Application-CK}, we have
\[
r^{0}>0,\quad r^0_s<0,\quad \forall (\theta,s)\in \R\times [-\delta,\delta]
\]
Therefore, 
\begin{equation}\label{0829-3}
G_\rho (\theta,r^0(\theta,s)s,0)=u_{\rho}(r^0(\theta,s)e(\theta))=\frac{1}{r^0_s(\theta,s)}<0.
\end{equation}
Therefore, for any $\theta\in \R$ and $s=0$, by \eqref{0829-2}, \eqref{0829-3} and the implicit theorem, we conclude that there exists $\varepsilon_{*}>0$, $0<\delta_{*}<\delta$ and a jointly real-analytic function $r^{\varepsilon}$ 
\[
r^{\varepsilon}(\theta,s)=r(\theta,s,\varepsilon),\quad (\theta,s,\varepsilon)\in (\theta-\delta_{*},\theta+\delta_{*})\times (-\delta_{*},\delta_{*})\times (-\varepsilon_{*},\varepsilon_{*}).
\]
such that
\[
G(\theta,r^{\varepsilon},s,\varepsilon)=0.
\]
And $r^{\varepsilon}$ is unique. By the compactness of set $[0,2\pi]$ and uniqueness of solution $r^{\varepsilon}$, we conclude that after decreasing $\delta_{*}$ and $\varepsilon_{*}$,
\[
u^{\varepsilon}(r^{\varepsilon}(\theta,s)e(\theta))=s,\quad (\theta,s,\varepsilon)\in [0,2\pi)\times (-\delta_{*},\delta_{*})\times (-\varepsilon_{*},\varepsilon_{*}).
\]
Moreover, $r^{\varepsilon}$ can be extended to $\R$ with $2\pi$ period and
\[
r^{\varepsilon}>0,\quad r_s^{\varepsilon}<0.
\]
After decreasing $\delta_{*}$, by the jointly real-analyticity of $r^{\varepsilon}$. we conclude that $r^{\varepsilon}\rightarrow r^{0}$ in $C^{k}$ uniformly in $[0,2\pi)\times [-\delta_{*},\delta_{*}]$ for every fixed $k$. 
\end{proof}

\begin{lemma}\label{3.2-2}
Let $r^\varepsilon$ is given by Lemma \ref{3.2-1}. For any $(\theta,s,\varepsilon)\in [0,2\pi)\times [-\delta_{*}.\delta_{*}]\times (-\varepsilon_{*},\varepsilon_{*})$, define
\[
\Ccurv^{\varepsilon}(\theta,s):=(r^{\varepsilon})^2+2(r^{\varepsilon}_\theta)^2-r^{\varepsilon} r^{\varepsilon}_{\theta\theta}.
\]
Then for every sufficiently small $\varepsilon>0$, we have
\[
\Ccurv^{\varepsilon}(0,0)<0,\quad \Ccurv^{\varepsilon}(\pi,0)<0.
\]
\end{lemma}

\begin{proof}[\bf Proof.]
Let
\[
f(\theta)=\partial_{\varepsilon} r^{\varepsilon}(\theta,0)|_{\varepsilon=0}.
\]
By the definition of $\Ccurv^{\varepsilon}$, we have
\begin{equation}\label{0829-4}
\dot{\Ccurv}(\theta,0)=\partial_{\varepsilon} \Ccurv^{\varepsilon} (\theta,0)|_{\varepsilon=0}=2r^{0}(\theta,0)f(\theta)+4r_\theta^{0}(\theta,0) f_{\theta}(\theta)-r^{0}(\theta,0)f_{\theta\theta}(\theta\theta)-f(\theta)r_{\theta\theta}^{0}.
\end{equation}
By Lemma \ref{Application-CK}, we know that
\[
r^{0}(\theta,0)=H(\theta)=10-2\cos 2\theta,\quad r^{0}_s=Q(\theta)=-H(\theta)+7\sin \theta -\sin 3\theta.
\]
Therefore, by \eqref{0829-4}, we have
\begin{equation}
\dot{\Ccurv}(\theta,0)=2Hf+4H_\theta f_{\theta}-Hf_{\theta\theta}-fH_{\theta\theta}.
\end{equation}
At $\theta_{*}\in \{0,\pi\}$, we have
\begin{equation}\label{0829-5}
\dot{\Ccurv}(\theta_{*},0)=8(f-f_{\theta\theta}).
\end{equation}
By Lemma \ref{3.2-1}, we have
\[
(u+\varepsilon w)(r^{\varepsilon}(\theta,s)e(\theta))=s.
\]
Differentiating above equality with respect to $\varepsilon$ and let $s=0$, $\varepsilon=0$, we have
\[
u^{0}_{\rho}f(\theta)+w(r^{0}(\theta,0)e(\theta))=u^{0}_{\rho}f(\theta)+w(H(\theta)e(\theta))=Q^{-1}(\theta)f(\theta)+w(H(\theta)e(\theta))=0.
\]
By \eqref{Auxiliary-func}, we have
\[
w(H(\theta)e(\theta))=-\int_{0}^{H(\theta)\sin \theta} e^{t^2/2} dt.
\]
For simplicity, let $W(\theta):=w(H(\theta)e(\theta))$.
Then
\[
W_\theta(0)=-8,\quad W_{\theta}(\pi)=8,\quad W_{\theta\theta}(0)=W_{\theta \theta}(\pi)=0.
\]
By \eqref{0829-5}, we have 
\begin{equation}
\dot{\Ccurv}(\theta_{*},0)=-512.
\end{equation}
By Lemma \ref{Application-CK}, if $\theta_{*}\in \{0,\pi\}$ and $s=0$, we know that
\begin{equation}
\Ccurv^{0}(\theta_{*},0)=0.
\end{equation}
Since $C^{\varepsilon}$ is analytic w.r.t $\varepsilon$, the Taylor expansion at the point $(\theta_{*},0,0)$ yields
\[
C^{\varepsilon}(\theta_{*},0)=-512\varepsilon+O(\varepsilon^2).
\]
We complete the proof.
\end{proof}

\subsection{Proof of Theorem \ref{T1}}
In the following, we give the proof of Theorem \ref{T1} under the assumptions $n=2$ and $p=2$.

\begin{proof}[\bf Proof of Theorem \ref{T1}.]
By Lemma \ref{Application-CK}, we have
\begin{equation}
m:=\mathop{\min}_{\theta\in [0,2\pi)} \min \{ \Ccurv^{0}(\theta,-\delta), \Ccurv^{0}(\theta,\delta) \}>0.
\end{equation}
By Lemma \ref{3.2-2}, there exist $\varepsilon_1$ such that for any $\varepsilon\in (0,\varepsilon_1) $
\begin{equation}\label{0829-7}
\Ccurv^{\varepsilon}(0,0)<0,\quad \Ccurv^{\varepsilon}(\pi,0)<0.
\end{equation}
By Lemma \ref{3.2-1}, we have
\[
\Ccurv^{\varepsilon}(\cdot,\pm \delta)\rightarrow \Ccurv^{0}(\cdot,\pm \delta)\quad \mathrm{uniformly\ in}\ [0,2\pi).
\]
Therefore, there exists $\varepsilon_2>0$ such that for any $\varepsilon\in (0,\varepsilon_2)$, we have
\begin{equation}\label{0829-6}
\Ccurv^{\varepsilon}(\theta,\pm \delta)>\frac{m}{2}>0,\quad  \forall\ \theta\in [0,2\pi).
\end{equation}
Choose
\[
0<\bar{\varepsilon}<\min\{\varepsilon_1,\varepsilon_2\}.
\]
Write
\[
\bar{u}=u^{\bar{\varepsilon}},\quad \bar{r}=r^{\bar{\varepsilon}},\quad \bar{\Ccurv}=\Ccurv^{\bar{\varepsilon}}.
\]
Let
\[
K_s=\{ \rho e(\theta):\ 0\leq \rho \leq \bar{r}(\theta,s),\ \theta\in [0,2\pi) \},\quad s\in [-\delta,\delta].
\]
Since $\bar{r}_s<0$, we have
\[
0\in K_{\delta} \Subset  K_{-\delta}.
\]
By \eqref{0829-6} and Lemma \ref{Application-CK}, $K_{\pm \delta}$ have real-analytic strictly convex boundaries. Moreover, \eqref{0829-7} implies that $K_0$ is not convex.

Define
\[
\Omega_0=\mathrm{int}\ K_{-\delta},\quad \Omega_1=\mathrm{int}\ K_{\delta},\quad \Omega=\Omega_0\setminus \overline{\Omega}_1
\]
Lemma \ref{3.2-1} gives
\[
\bar{u}=-\delta\quad \mathrm{on}\ \partial \Omega_0,\quad \bar{u}=\delta\quad \mathrm{on}\ \partial \Omega_1.
\]
Therefore, 
\[
\upsilon=\frac{\bar{u}+\delta}{2\delta}
\]
satisfies \eqref{eq:OU-problem-intro}. The uniqueness of the Dirichlet solution follows from the weak maximum
principle for linear uniformly elliptic equations with bounded first-order coefficients.

Finally, the fact $\upsilon=0.5$ implies that $\bar{u}=0$. Therefore, $\Omega_{0.5}$ is not convex. This proves the theorem.
\end{proof}

\section{Counterexamples in higher dimensions and nonlinear extension}

\subsection{Counterexample in higher dimensions for \texorpdfstring{$p=2$}{p=2}}
In this section, we construct a counterexample in higher dimensions and $p=2$.
In $\mathbb{R}^n$, the position vector can be parameterized in term of the radial function as follows
\begin{equation}\label{eq:X-def2}
X(\omega,s)=r(\omega,s)\omega,\quad \omega\in \mathbb{S}^{n-1},
\end{equation}
where $r(\omega,s)\in C^2(\mathbb{S}^{n-1}\times (-\delta,\delta))$ $(\delta>0)$ satisfies
\begin{equation}\label{eq:r-signs-preliminary2}
 r>0,
 \qquad
 r_s <0.
\end{equation}
Locally, \eqref{eq:r-signs-preliminary2} allows us to define a function $u$ by
\begin{equation}\label{eq:v-def2}
u \bigl(r(\omega,s)\omega\bigr)=s.
\end{equation}
Writing a point as $x=\rho \omega$, the defining relation is
\begin{equation}\label{0903-2}
\rho=r\bigl(\omega,u(\rho,\omega)\bigr).
\end{equation}

Let $g_{ij}$ be the standard metric on $\mathbb{S}^{n-1}$ and let $\nabla$, $\nabla^2$ and $\Delta_{\mathbb{S}^{n-1}}$ denote the gradient, Hessian and Laplacian on $\mathbb{S}^{n-1}$. On the hypersurface $X(\omega)=r(\omega,s)\omega$, $\omega\in \mathbb{S}^{n-1}$, the outward unit normal and the second fundamental form are
\[
\nu=\frac{r\omega-\nabla r}{\sqrt{r^2+|\nabla r|^2}},\quad \Pi=\frac{\Ccurv(r)}{\sqrt{r^2+|\nabla r|^2}}.
\]
where 
\begin{equation}\label{0903-11}
\Ccurv(r):=r^2 g_{ij}+2r_{i}r_{j}-r\nabla^2 r.
\end{equation}
Thus the fact that matrix $\Ccurv(r)$ is non-negative definite means that the hypersurface is convex. While
$C(r)<0$ at one point excludes convexity.

At the point under consideration, we choose geodesic normal coordinates on $\mathbb S^{n-1}$, so that $g_{ij}=\delta_{ij}$ and $\Gamma_{ij}^{k}=0$ at that point. Hence, at this point, the covariant Hessian agrees with the ordinary second derivatives.

\begin{lemma}\label{prop:evolution2}
Assume that \eqref{eq:r-signs-preliminary2} holds. The function $u$ defined by \eqref{eq:v-def2} satisfies
\begin{equation}\nonumber
 -\Delta u+x\cdot\nabla u=0,
\end{equation}
if and only if $r$ satisfies
\begin{equation}\nonumber
 (r^2+|\nabla r|^2)r_{ss}
 -2 r_s(\nabla r, \nabla r_{s})
 +r_s^2\Delta_{\mathbb{S}^{n-1}} r
 +(r^3-(n-1)r)r_s^2=0.
\end{equation}
Equivalently, 
\begin{equation}\label{eq:evolution2}
 r_{ss}
 =\frac{2 r_s(\nabla r,\nabla r_{s})
       -r_s^2 \Delta_{\mathbb{S}^{n-1}} r
       -(r^3-(n-1)r)r_s^2}
      {r^2+|\nabla r|^2}.
\end{equation}
\end{lemma}

\begin{proof}[\bf Proof.]
Using spherical coordinates, we have
\begin{equation}\nonumber
\Delta u= u_{rr}+\frac{n-1}{r}u_r+\frac{1}{r^2}\Delta_{\mathbb{S}^{n-1}} u,\quad (x,\nabla u)=ru_r
\end{equation}
Thus, we deduce that
\begin{equation}\label{0903-3}
u_{rr}+\frac{n-1}{r}u_r+\frac{1}{r^2}\Delta_{\mathbb{S}^{n-1}} u-ru_r=0.
\end{equation}
By \eqref{eq:v-def2}, we have
\begin{equation}\label{0903-4}
u_r=\frac{1}{r_s},\quad u_{rr}=-\frac{r_{ss}}{r_s^3}.
\end{equation}
By \eqref{0903-2}, we have
\begin{equation}\label{0903-5}
\nabla_{\mathbb{S}^{n-1}} u=-\frac{1}{r_s}\nabla_{\mathbb{S}^{n-1}} r,\quad 
\Delta_{\mathbb{S}^{n-1}} u=-\frac{1}{r_s}\Delta_{\mathbb{S}^{n-1}} r+2\frac{1}{r_s^2}(\nabla r,\nabla r_s)-\frac{r_{ss}}{r_s^3}|\nabla_{\mathbb{S}^{n-1}}r|^2.
\end{equation}
Here $\nabla r$ denotes the derivative of $r(\omega,s)$ with respect to its first variable with $s$ fixed. By \eqref{0903-3}, \eqref{0903-4} and  \eqref{0903-5}, we complete the proof.
\end{proof}

To apply Cauchy-Kowalevski Theorem, we choose Cauchy data on $\omega=(\omega_1,\ldots,\omega_n)\in \mathbb{S}^{n-1}$ as follows 
\begin{equation}\label{0830-1}
H(\omega)=H(\omega_2)=A\left( 1+\frac{\omega_2^2}{2} \right),\quad Q(\omega)=Q(\omega_2)=-H(\omega_2)+4\omega_2(1+\omega_2^2),
\end{equation}
where 
\begin{equation}\label{0904-1}
A>\max\left\{ \frac{16}{3},\sqrt{\frac{32n}{23}}\right\}
\end{equation}
is a constant. For simplicity, we define $z(\omega):=\omega_2$. Obviously, $H(z)\geq A$ and $Q(z)<0$, for all $z\in [-1,1]$.

\begin{lemma}
There exists a unique solution $r(\omega,s)$ which satisfies equation \eqref{eq:evolution} 
with Cauchy data
\begin{equation}\label{0903-15}
 r(\omega,0)=H(\omega_2),
 \qquad
 r_s(\omega,0)=Q(\omega_2),\quad \forall\ \omega=(\omega_1,\cdots,\omega_n)\in \mathbb{S}^{n-1}.
\end{equation}
\begin{enumerate}
\item After decreasing $s_0$ if necessary,
\[
 r>0,
 \qquad
 r_s<0
 \quad\text{in } \mathbb{S}^{n-1}\times [-s_0,s_0].
\]
Moreover, the solution is symmetric with respect to the axis $e_2$, that is,
\begin{equation}\label{0831-1}
r(\omega,s)=r(\omega_2,s),\quad s\in [-s_0,s_0].
\end{equation}

\item If $s=0$, $\Ccurv(\omega,0)$ is a positive semi-definite matrix on $\mathbb{S}^{n-1}$.  The eigenvalues of $\Ccurv(\omega,0)$ on $\mathbb{S}^{n-1}$ are 
\[
\lambda (\omega,0)=A^2\left( 1+\frac{z^2}{2}\right) \left( 1+\frac{3z^2}{2}\right)>0
\]
with multiplicity $n-2$, and 
\[
\Lambda(\omega,0)=\frac{3}{4}A^2\left(6-z^2 \right)z^2\geq 0.
\]
with multiplicity $1$.

\item If $s=0$, $\omega_{*}=(\omega_1,\cdots,\omega_n)\in \mathbb{S}^{n-1}$ with $\omega_{2}=0$, we have
\[
\Lambda(\omega_{*},0)=0,\quad  \Lambda_z(\omega_{*},0)= \Lambda_s(\omega_{*},0)=\Lambda_{zs}(\omega_{*},0)=0
\]
and
\[
\Lambda_{zz}(\omega^{*},0)=9A^2,\quad\Lambda_{ss}(\omega_{*},0)=23A^2-32n+64>0. 
\]
\end{enumerate}
\end{lemma}

\begin{proof}[\bf Proof.]

Apply Cauchy-Kowalevski theorem to equation \eqref{eq:evolution2} 
with the above Cauchy data, the existence and uniqueness of solution is proved. Moreover, since the equation \eqref{eq:evolution2} are invariant under rotations, the Cauchy data are symmetric w.r.t the axis $e_2$, the uniqueness of solution yields
\begin{equation}
r(\omega,s)=r(\omega_2,s),\quad s\in [-s_0,s_0].
\end{equation}
where $\omega=(\omega_1,\ldots,\omega_n)\in \mathbb{S}^{n-1}$. we define
\[
z(\omega)=\omega_2.
\]
Fix $\omega\in \mathbb{S}^{n-1}$ with $|z|<1$, since $\nabla z=e_2-z\omega $, we deduce that $|\nabla z|\ne 0$. Then we choose the orthonormal basis on $T_{\omega}\mathbb{S}^{n-1}$ as
\[
E_1=\frac{\nabla z}{|\nabla z|}=\frac{\nabla z}{\sqrt{1-z^2}}, E_2,\cdots, E_{n-1}.
\]
Here at the point $\omega$,  $g_{ij}=\delta_{ij}$ and $\Gamma^{k}_{ij}=0$. Consequently, we deduce that
\begin{equation}\label{0903-14}
z_i z_j=\text{diag}(1-z^2,0,\ldots,0),\quad \nabla^2 z=-z\delta_{ij}.
\end{equation}
By \eqref{0831-1}, we have
\begin{equation}
\nabla r=r_z \nabla z,\quad \nabla^2 r=(r_{ij})=r_{zz}z_{i}z_{j}-zr_{z}\delta_{ij}.
\end{equation}
In particular, if $s=0$, $r(\omega,0)=H(z)$, we have
\begin{equation}\label{0903-17}
\nabla r=H_z(e_2-z\omega),\quad \nabla^2 r=(r_{ij})=H_{zz}z_{i}z_{j}-zH_{z}\delta_{ij}.
\end{equation}
By \eqref{0903-11}, we have
\begin{equation}\label{0903-12}
\Ccurv(\omega,s)=\left(r^2+zrr_{z}\right) \delta_{ij}+\left(2r_z^2-rr_{zz}\right) z_i z_j.
\end{equation}
In particular, if $s=0$, we have
\begin{equation}\label{0903-13}
\Ccurv(\omega,0)=(H^2+zHH_z)\delta_{ij}+(2H^2_z-HH_{zz})z_i z_{j}.
\end{equation}
Therefore, by \eqref{0903-14}, the eigenvalues of $\Ccurv(\omega,0)$ are
\begin{equation}\label{0903-7}
\lambda(\omega,0)=H^2+zHH_z,
\end{equation}
with multiplicity $n-2$, and
\begin{equation}\label{0903-8}
\Lambda(\omega,0)=H^2+zHH_z+(1-z^2)(2H^2_z-HH_{zz}).
\end{equation}
If $z=\pm 1$, we have $\nabla r=0$. Then 
\[
C(\omega,0)=(H^2+zHH_z)\delta_{ij}.
\]
By \eqref{0830-1}, we have
\begin{equation}\label{0903-9}
\left\{
\begin{aligned}
&H_z=Az,\quad H_{zz}=A,\quad H_{zzz}=0;\\
&Q_z=12z^2-Az+4,\quad Q_{zz}=24z-A,\quad Q_{zzz}=24,\quad Q_{zzz}=0.
\end{aligned}
\right.
\end{equation}
In particular, if $z=0$, we have
\begin{equation}\label{0903-10}
\left\{
\begin{aligned}
&H=A,\quad H_z=0,\quad H_{zz}=A,\quad H_{zzz}=0,\quad H_{zzzz}=0;\\
&Q=-A,\quad Q_z=4,\quad Q_{zz}=-A,\quad Q_{zzz}=24,\quad Q_{zzzzz}=0.
\end{aligned}
\right.
\end{equation}
Therefore, by \eqref{0830-1}, \eqref{0903-7}, \eqref{0903-8} and \eqref{0903-9}, we have
\begin{equation}
\left\{
\begin{aligned}
&\lambda(\omega,0)=A^2\left( 1+\frac{z^2}{2}\right) \left( 1+\frac{3z^2}{2}\right)>0,\\
&\Lambda(\omega,0)=\frac{3}{4}A^2\left(6-z^2 \right)z^2\geq 0.
\end{aligned}
\right.
\end{equation}
This proves $(2)$.

If $s=0$, $\omega_{*}=(\omega_1,\cdots,\omega_n)\in \mathbb{S}^{n-1}$ with $\omega_{2}=0$. Obviously, $z(\omega)=0$. By \eqref{0903-8} we have
\begin{equation}
\left\{
\begin{aligned}
& \Lambda(\omega_{*},0)=0,\\
&\Lambda_{z}(\omega_{*},0)=\frac{3}{4}A^2\left( 12z-4z^3\right)= 0,\\
&\Lambda_{zz}(\omega_{*},0)=\frac{3}{4}A^2\left( 12-12z^2\right)=9A^2.
\end{aligned}
\right.
\end{equation}
By \eqref{0903-12}, we have
\begin{equation}
\Lambda(\omega,s)=r^2+zrr_z+(1-z^2)(2r^2_z-rr_{zz}).
\end{equation}
A direct calculation yields that
\begin{equation}\label{0903-16}
\left\{
\begin{aligned}
&\Lambda_{s}(\omega,s)=2rr_{s}+zr_sr_z+zrr_{sz}+(1-z^2)(4r_zr_{sz}-r_sr_{zz}-rr_{szz}),\\
&\Lambda_{ss}(\omega,s)=2rr_{ss}+2r_s^2+zr_zr_{ss}+2zr_{s}r_{sz}+zrr_{ssz}\\
&\qquad \qquad \quad+(1-z^2)(4r_{sz}^2+4r_zr_{ssz}-2r_{s}r_{szz}-r_{ss}r_{zz}-rr_{sszz}).
\end{aligned}
\right.
\end{equation}
Then, by \eqref{0903-15}, \eqref{0903-10} and \eqref{0903-16}, we have
\begin{equation}
\left\{
\begin{aligned}
&\Lambda_{s}(\omega_{*},0)=0,\quad \Lambda_{sz}(\omega_{*},0)=0,\\
&\Lambda_{ss}(\omega_{*},0)=Ar_{ss}-Ar_{sszz}+64.
\end{aligned}
\right.
\end{equation}
By Lemma \ref{prop:evolution2}, \eqref{0903-17} and \eqref{0903-15}, we have
\begin{equation}\label{0903-6}
r_{ss}=\frac{2(1-z^2)QH_z Q_z-Q^2(1-z^2)H_{zz}+(n-1)zQ^2H_z-(H^3-(n-1)H)Q^2}{H^2+(1-z^2)H^2_z}.
\end{equation}
Differentiating \eqref{0903-6} with respect to $z$, then let $z=0$, by \eqref{0903-10}, we have
\begin{equation}
\left\{
\begin{aligned}
&r_{ss}=-A^3+(n-2)A,\quad r_{ssz}=8(A^2-(n-1));\\
&r_{sszz}=-A^3+(n-25)A+32nA^{-1}.
\end{aligned}
\right.
\end{equation}
Then, we deduce that
\[
\Lambda_{ss}(\omega_{*},0)=23A^2-32n+64.
\]
By \eqref{0904-1}, we have $\Lambda_{ss}(\omega_{*},0)>0$.
\end{proof}

\begin{remark}
When $n=2$, one may take $A=8$ and $z=\sin \theta$, $\theta\in [0,2\pi)$, so the Cauchy data coincide with those used in the planar construction. The remaining argument of the proof is analogous to that of Section 3.2 and Section 3.3, we omit its proof. The key point in higher dimensions is that there exists a negative eigenvalue of the second fundamental form at one point, which is sufficient to exclude convexity of the corresponding level set. 
\end{remark}

\subsection{Counterexample for Gaussian  \texorpdfstring{$p$}{p}-Capacity} 

Here, we only give the sketch of the proof, since all the proof is very similar to Section 3 and Section 4.1. Here the main difficulty is that the $p$-Laplacian operator is no longer linear, the linear perturbation is invalid in Section 3.2. We instead perturb the analytic Cauchy data and using the continuous dependence of solution with respect the Cauchy data. See Lemma \ref{0909-1}. We therefore construct a convex ring such that the solution of \eqref{eq:OU-problem-intro} has non-convex level sets.

This view is more simpler and it works. we give more details as follows.

First, similar to Section 3 and Section 4.1, we can prove that under the assumption \eqref{eq:r-signs-preliminary2} for the radial function, the function $u$ defined by \eqref{eq:v-def} satisfies $-L_{p} (u)=0$
if and only if $r$ satisfies the following equation
\begin{align*}
&(p-1)(r^2+|\nabla r|^2)r_{ss}
 -2(p-1) r_s(\nabla r, \nabla r_{s})+r_s^2\Delta_{\mathbb{S}^{n-1}} r\\
 &\quad +(r^3-(n-1)r)r_s^2+ (p-2)\frac{r^2_s}{r^2+|\nabla r|^2}\left( r|\nabla r|^2+\nabla^2r (\nabla r,\nabla r)\right)
=0.
\end{align*}
Second, let
\[
z=z(\omega)=\omega_2,\quad \omega=(\omega_1,\omega_2.\ldots,\omega_n)\in \mathbb{S}^{n-1}.
\]
We choose the Cauchy data
\[
H_b(z)=A\left( 1+\frac{b}{2} z^2\right),\quad  Q(z)=-H(z)+4z(1+z^2),\quad z\in [-1,1],
\]
where $b>1$ is a constant, and $A>\max \left\{ \frac{16}{3},\sqrt{\frac{32n}{23}}\right\}$. By a direct calculation, we can deduce that $\Ccurv(\omega,0)$ has a negative eigenvalue at some points, that is, 
\[
\lambda(\omega_{*},0)=A^2(1-b)<0
\]
with multiplicity $1$, and
\[
\lambda(\omega_{*},0)=A^2>0
\]
with multiplicity $n-2$, where $\omega_{*}=(\omega_1,0,\omega_3,\ldots,\omega_n)\in \mathbb{S}^{n-1}$. Therefore, if $b\leq 1$, the hypersurface has the points with non-positive eigenvalue. 

Third, choose perturbed Cauchy data ad follows
\[
H^{\varepsilon}(z)=A\left( 1+\frac{1+\varepsilon}{2} z^2\right),\quad Q^\varepsilon(z)=-H^{\varepsilon}(z)+4z(1+z^2), \quad \varepsilon \in [0,1].
\]
Applying the Cauchy-Kowalevski Theorem to the radial-graph equation with the above Cauchy date, the existence and uniqueness of solution can be proved by Proposition \ref{0911}. Moreover, the continuous dependence of solution $r^{\varepsilon}$ with respect the Cauchy data is proved in \cite[Theorem~3 and Remark~3]{Walter1985}. See also Proposition \ref{0909}. Indeed, the convergence of the Cauchy data is uniform with respect to $\varepsilon$, Walter's
contraction estimate therefore yields $C^k$ convergence. See the following Lemma \ref{0909-1}.

Finally, by an approximation argument, we can directly prove that there exists a sufficiently small $\varepsilon\in (0,1)$ such that the eigenvalues of $\Ccurv(\omega,s)$ on the both boundary hypersurface are positive. This means that the boundary hypersurface is strong convex. However, when $s=0$, the hypersurface has negative eigenvalue at some points. This implies that the hypersurface at $s=0$ is non-convex. 

In the following, we present the continuous dependence od solution with respect to the perturbed Cauchy data.

\begin{lemma}\label{0909-1}
There exist $\varepsilon_0>0$ and $\delta_0>0$, such that for any $\varepsilon\in [0,\varepsilon_0)$, the unique solution $r^{\varepsilon}$ exists and is well-defined in $\mathbb{S}^{n-1}\times [-\delta_0,\delta_0]$. Moreover, $r^{\varepsilon}\rightarrow r^0$ in $C^{k}(\mathbb{S}^{n-1}\times [-\delta_0,\delta_0])$ for any fixed $k$. Let
$$
\Ccurv^{\varepsilon}(\omega,s):=(r^{\varepsilon})^2+2|\nabla r^{\varepsilon} |^2-r^{\varepsilon} \nabla^2 r^{\varepsilon}.
$$
We conclude that 
\[
\|\Ccurv^{\varepsilon}- \Ccurv^{0}\|:=\mathop{\sup}_{\substack{(\omega,s)\in \\ \mathbb{S}^{n-1}\times [-\delta_0,\delta_0]} } \mathop{\sup}_{\substack{ \xi \in T_{\omega}\mathbb{S}^{n-1}\\|\xi|_g  \leq 1} }  | (\Ccurv^{\varepsilon}- \Ccurv^{0})(\omega,s)[\xi,\xi]|\rightarrow 0,\quad as\ \varepsilon \rightarrow 0. 
\]
\end{lemma}

\begin{proof}[\bf Proof.]
Let $\omega=(\omega_1,\ldots,\omega_n)\in \mathbb{S}^{n-1}$, and let $y=(y^1,\ldots,y^{n-1})$ be a local coordinate system
on $\mathbb S^{n-1}$, so that $\omega=\omega(y)$. We write
\[
\partial_i=\frac{\partial}{\partial y^i},
\qquad 1\leq i\leq n-1.
\]
Define
\[
 U^{\varepsilon}=\left(
r^\varepsilon,r^\varepsilon_s,\partial_1 r^\varepsilon,\ldots,\partial_{n-1}r^\varepsilon
\right)^{\mathsf T}.
\]
Then the radial-graph equation associated with
the Gaussian $p$-capacitary equation is equivalent to the
following first-order quasi-linear system. 
\begin{equation}\label{0907-1}
\partial_sU^\varepsilon
=
\sum_{\ell=1}^{n-1}
B^\ell(y,U^\varepsilon)\,
\partial_\ell U^\varepsilon
+
C(y,U^\varepsilon).
\end{equation}
Here $B^\ell$ is an $(n+1)\times(n+1)$ matrix. 
All entries of $B^\ell$ vanish except
\[
\left\{
\begin{aligned}
\left(B^\ell\right)_{2,2}
={}&
\frac{
2r_s^\varepsilon
g^{a\ell}\partial_a r^\varepsilon
}{
(r^\varepsilon)^2
+
g^{ab}\partial_a r^\varepsilon\partial_b r^\varepsilon
},
\\[3mm]
\left(B^\ell\right)_{2,j+2}
={}&
-\frac{
(r_s^\varepsilon)^2g^{\ell j}
}{
(p-1)
\left[
(r^\varepsilon)^2
+
g^{ab}\partial_a r^\varepsilon\partial_b r^\varepsilon
\right]
}
\\
&-
\frac{
(p-2)(r_s^\varepsilon)^2
\left(g^{\ell a}\partial_a r^\varepsilon\right)
\left(g^{jb}\partial_b r^\varepsilon\right)
}{
(p-1)
\left[
(r^\varepsilon)^2
+
g^{ab}\partial_a r^\varepsilon\partial_b r^\varepsilon
\right]^2
},
\qquad 1\leq j\leq n-1,
\\[3mm]
\left(B^\ell\right)_{i+2,2}
={}&
\delta_i^\ell,
\qquad 1\leq i\leq n-1,
\end{aligned}
\right.
\]
where $1\leq i,j,\ell\leq n-1$. 
The vector $C(y,U^\varepsilon)=(C_1,\ldots,C_{n+1})$ has only two non-zero
components. More precisely, all its components vanish except
\[
\left\{
\begin{aligned}
C_1
={}&r_s^\varepsilon,
\\[3mm]
C_2
={}&
\frac{
(r_s^\varepsilon)^2
g^{ij}\Gamma_{ij}^{k}\partial_k r^\varepsilon
}{
(p-1)
\left[
(r^\varepsilon)^2+
g^{ij}\partial_i r^\varepsilon\partial_j r^\varepsilon
\right]
}
\\[2mm]
&-
\frac{
(p-2)(r_s^\varepsilon)^2
g^{ia}\partial_a r^\varepsilon
\left[
2r^\varepsilon\partial_i r^\varepsilon
+
(\partial_i g^{jk})
\partial_j r^\varepsilon\partial_k r^\varepsilon
\right]
}{
2(p-1)
\left[
(r^\varepsilon)^2+
g^{jk}\partial_j r^\varepsilon\partial_k r^\varepsilon
\right]^2
}
\\[2mm]
&-
\frac{
\left[
(r^\varepsilon)^3-(n-1)r^\varepsilon
\right]
(r_s^\varepsilon)^2
}{
(p-1)
\left[
(r^\varepsilon)^2+
g^{ij}\partial_i r^\varepsilon\partial_j r^\varepsilon
\right]
},
\\[3mm]
C_{i}
={}&0, \quad \qquad 3\leq i\leq n+1.
\end{aligned}
\right.
\]
Repeated indices are summed from $1$ to $n-1$.

Since $g^{ij}$ and $\Gamma_{ij}^{k}$ are real analytic in each
real-analytic coordinate chart of $\mathbb S^{n-1}$, the
matrices $B^\ell(y,U)$ $(\ell=1,\ldots,n-1)$ and the vector $C(y,U)$ are real
analytic if
\[
(r^\varepsilon)^2+g^{ij}\partial_i r^{\varepsilon}\,\partial_j r^\varepsilon\neq0.
\]
Moreover, there exists a constant $c_0>0$ independent of $\varepsilon$ such that
\[
\bigl(H^\varepsilon\bigr)^2
+
\left|
\nabla_{\mathbb S^{n-1}}H^\varepsilon
\right|^2
\geq c_0
\qquad\text{on }\mathbb S^{n-1}.
\]
Therefore, for any $\varepsilon \in (0,\varepsilon_0)$, when $s=0$, one has
\[
\left|
(r^{\varepsilon})^2+g^{ij}\partial_i r^\varepsilon\,\partial_j r^\varepsilon
\right|
\geq \frac{c_0}{2},\quad \mathrm{on}\ \mathbb{S}^{n-1}.
\]
Thus $B^\ell$ and $C$, together with their derivatives with
respect to $y$ and $U^{\varepsilon}$, satisfy the uniform bounds and
Lipschitz estimates required in
Walter~\cite[Theorem~3]{Walter1985}. Since $\mathbb{S}^{n-1}$ is compact, by Proposition \ref{0909}, we conclude that there exists a common region $\mathbb{S}^{n-1}\times [-\delta,\delta]$ such that for any $\varepsilon \in (0,\varepsilon_0)$, there exists a unique real-analytic solution $r^{\varepsilon}$ in $\mathbb{S}^{n-1}\times [-\delta,\delta]$. Moreover, the solution $U^{\varepsilon}\rightarrow U^0$ as $\varepsilon \rightarrow 0$ uniformly in $\mathbb{S}^{n-1}\times[-\delta,\delta]$. This implies the convergence of $r^\varepsilon$, $r_s^\varepsilon$ and $\nabla r^{\varepsilon}$.

Finally, differentiating \eqref{0907-1} with respect to $y^i$ $(i=1,\cdots,n-1)$. applying the same argument to the
resulting system, we obtain that $\nabla^2 r^\varepsilon \rightarrow \nabla^2 r^0 $ as $\varepsilon \rightarrow 0$ uniformly in $\mathbb{S}^{n-1}\times [-\delta,+\delta]$. This implies that $\Ccurv^{\varepsilon}\rightarrow \Ccurv^{0}$ uniformly in $\mathbb{S}^{n-1}\times [-\delta,\delta]$ as $\varepsilon \rightarrow 0$.

The uniform convergence of the higher-order derivatives of $r^\varepsilon$ follows by iterating the same argument.
\end{proof}

\section*{Acknowledgment}
The author would like to thank Professors Andrea Colesanti, Pengfei Guan and Paolo Salani for their patient guidance and warm encouragement. Much of what eventually led to this paper grew out of my time in Florence. I would therefore like this paper to be, in some sense, a tribute to the those in Florence, as well as those closely connected with Florence, who have helped and supported me over the years.

The author acknowledge ChatGPT for performing the complex computations that led to the discovery of the Cauchy data.

\end{document}